\documentclass{amsart}
\usepackage{amsmath}
\usepackage{amssymb}

\usepackage[numbers,sort&compress]{natbib}

\usepackage{graphicx}
\usepackage{epstopdf}

\newtheorem{theorem}{Theorem}

\newtheorem{Proposition}[theorem]{Proposition}
\newtheorem{corollary}[theorem]{Corollary}
\theoremstyle{definition}
\newtheorem{definition}[theorem]{Definition}
\newtheorem{example}[theorem]{Example}

\theoremstyle{Remark}
\newtheorem{remark}[theorem]{Remark}

\numberwithin{equation}{section}    

\begin{document}
\title[On the numerical radius parallelism and Birkhoff orthogonality]{On the numerical radius parallelism and the numerical radius Birkhoff orthogonality}

\author[J.Y. BI]{Jiaye Bi}
\address{Department of Mathematics,
	Sun Yat-sen University, Guangzhou, 510275, P. R. China
}
\curraddr{}
\email{bijiaye@mail2.sysu.edu.cn}

\author[H.Y. Xie ]{Huayou Xie}
\address{Department of Mathematics,
	Sun Yat-sen University, Guangzhou, 510275, P. R. China
}
\email{xiehy33@mail2.sysu.edu.cn}

\thanks{$*$ Corresponding author}
\author[Y.J. LI]{Yongjin Li$^{*}$}
\address{Department of Mathematics,
	Sun Yat-sen University, Guangzhou, 510275, P. R. China
}
\curraddr{}
\email{stslyj@mail.sysu.edu.cn}


\date{}

\dedicatory{}

\begin{abstract}
In this paper, we generalize the notions of numerical radius parallelism and numerical radius Birkhoff orthogonality, originally formulated for operators on Hilbert spaces, to operators on normed spaces. We then proceed to demonstrate their fundamental properties. Notably, our findings reveal that numerical radius parallelism lacks transitivity, and numerical radius Birkhoff orthogonality is neither left nor right additive. Additionally, we offer characterizations for both concepts. Furthermore, we establish a connection between numerical radius parallelism and numerical radius Birkhoff orthogonality.
\end{abstract}
\maketitle

\section{Introduction}
Throughout this paper, we assume that $X$ is a normed space over the field $\mathbb{F}\in\{\mathbb{R},\mathbb{C}\}$ of dimension at least 2, $B_X,S_X$ denote the closed unit ball and the unit sphere of $X$, respectively, $\mathbb{T}:=\{\lambda\in\mathbb{F}:|\lambda|=1\}$, and $\mathbb{B}(X)$ denote the normed space of all bounded linear operators on $X$.

The concept of \emph{norm parallelism}, first introduced by Seddik \cite{Seddik}, defines an element $x\in X$ as being \emph{norm parallel} to $y\in X$, denoted as $x\parallel y$, if there exists $\lambda\in \mathbb{T}$ such that the equality $\|x+\lambda y\|=\|x\|+\|y\|$ is satisfied. This notion generalizes the concept of linear dependence, as elements that are linearly dependent are necessarily norm parallel. Conversely, the implication holds true only in the context of strictly convex spaces, as established in \cite[Theorem 2.8]{Mal}. Zamani and Moslehian \cite{ContinuousFunctionsSpaces,Zamani} contributed significantly to the characterizations of the norm parallelism in Hilbert $C^*$-modules. In particular, they established a relation between norm parallelism and Birkhoff orthogonality \cite{Birkhoff}. Recall that a vector $x$ is said to be Birkhoff orthogonal to another vector $y$, denoted by $x\perp_B y$,  if and only if $\|x+\alpha y\|\geq\|x\|$ holds for all $\alpha\in\mathbb{R}$.
\begin{theorem}\rm{\cite[Theorem\;2.4]{Zamani}}\label{T-1}
	Let $X$ be a normed space. For any $x,y\in X$, there exists $\lambda\in \mathbb{T}$ such that the following statements are equivalent:
	
	$\mathrm{(i)}$ $x\parallel y$.
	
	$\mathrm{(ii)}$ $x\perp_B \|y\|x+\lambda\|x\|y$.
	
	$\mathrm{(iii)}$ $y\perp_B \|x\|y+\bar{\lambda}\|y\|x$.
\end{theorem} 
Furthermore, Zamani \cite{ContinuousFunctionsSpaces} conducted an investigation into the characterization of norm parallelism within specific spaces of continuous functions. Wojcik \cite{Wojcik} offered characterizations of norm parallelism for bounded linear operators between normed spaces and introduced a fascinating application pertinent to the invariant subspace problem. For those seeking a more profound understanding of norm parallelism, it is recommended to consult the references \cite{Mehrazin,Wojcik2,Zamani3}.

In Section \ref{s2}, we introduce the concept of \emph{numerical radius parallelism} in $\mathbb{B}(X)$ and establish some characterizations. Additionally, we show that numerical radius parallelism does not coincide with the linear dependence and lacks transitivity.

In Section \ref{s3}, we introduce the concept of \emph{numerical radius Birkhoff orthogonality} in $\mathbb{B}(X)$ and offer a characterization. Furthermore, we demonstrate that numerical radius Birkhoff orthogonality is neither additive on the left nor on the right. Finally, we provide a characterization of numerical radius parallelism in terms of numerical radius Birkhoff orthogonality.

\section{Numerical radius parallelism in $\mathbb{B}(X)$}\label{s2}
The \emph{numerical radius} of $T\in\mathbb{B}(H)$, where $(H,\left\langle\cdot,\cdot\right\rangle )$ is a complex Hilbert space, is given by
$$\omega(T):=\sup\{|\left\langle Tx,x \right\rangle|:x\in S_H \}.$$
The \emph{numerical radius} of bounded linear operators on a normed space $X$ (cf. \cite{Duncan}) is the semi-norm defined on $\mathbb{B}(X)$ by  
$$v(T):=\sup\{|x^*(Tx)|:x\in S_X,x^*\in J(x)\} \quad(T\in\mathbb{B}(X)),$$
where 
$J(x):=\{x^*\in S_{X^*}:x^*(x)=\|x\|\}$
is the \emph{duality mapping} for $x$.
Clearly, for any $x\in S_H$, there exist a unique $x^*\in S_{H^*}$ such that $x^*(x)=1$, that is $x^*=\left\langle \cdot,x \right\rangle $. Consequently, for $T\in\mathbb{B}(H)$, we have $v(T)=\omega(T)$. This demonstrates that the concept of the numerical radius of bounded operators on a Hilbert space generalizes to the numerical radius of bounded operators on a normed space.

In the context of $\mathbb{B}(X)$, the alternative Daugavet equation \cite{Martin}
$$\max_{\lambda\in \mathbb{T}}\|I+\lambda T\|=1+\|T\|$$ is a particular case of norm parallelism. In \cite{Duncan}, it was shown that, the previous equation holds if and only if $v(T)=\|T\|$.

Mehrazin $\mathit{et\;al.}$ \cite{Mehrazin} introduced the concept of numerical radius parallelism for Hilbert space operators and established a characterization, as follows.
\begin{definition}\cite[Definition 1.1]{Mehrazin}
	Let $H$ be a complex Hilbert space. An element $T \in \mathbb{B}(H)$ is called \emph{numerical radius parallel} to another  
	element $S \in \mathbb{B}(H)$, denoted by $T \parallel_{\omega} S$, if  
	$\omega(T + \lambda S) = \omega(T) + \omega(S)$ for some $\lambda \in \mathbb{T}$.
\end{definition}

\begin{theorem}\rm{\cite[Theorem 2.2]{Mehrazin}}\label{Mehrazin}
	Let $(H,\left\langle\cdot,\cdot\right\rangle )$ be a complex Hilbert space and $T, S \in \mathbb{B}(H)$. Then the following conditions are equivalent:  
	
	$\mathrm{(i)}$ $T \parallel_{\omega} S$.  
	
	$\mathrm{(ii)}$ There exists a sequence of unit vectors $\{x_n\}$ in $H$ such that  
	\[  
	\lim_{n \to \infty} |\langle Tx_n, x_n \rangle\langle Sx_n, x_n \rangle| = \omega(T)\omega(S).  
	\]  
	In addition, if $\{x_n\}$ is a sequence of unit vectors in $H$ satisfying $\mathrm{(ii)}$, then it also  
	satisfies  
	\[  
	\lim_{n \to \infty} |\langle Tx_n, x_n \rangle| = \omega(T) \quad \text{and} \quad \lim_{n \to \infty} |\langle Sx_n, x_n \rangle| = \omega(S).  
	\]
\end{theorem}

We now extend the concept of numerical radius parallelism, which was originally defined on $\mathbb{B}(H)$, to $\mathbb{B}(X)$.
\begin{definition}
	An element $T \in \mathbb{B}(X)$ is called \emph{numerical radius parallel} to another  
	element $S \in \mathbb{B}(X)$, denoted by $T \parallel_{v} S$, if  
	$v(T + \lambda S) = v(T) + v(S)$ for some $\lambda \in \mathbb{T}$.
\end{definition}
\begin{remark}
	Since $A\mapsto v(A)$ is a semi-norm on $\mathbb{B}(X)$, the condition $T \parallel_{v} S$ holds whenever $v(T + \lambda S) \geq v(T) + v(S)$ for some $\lambda \in \mathbb{T}$. In particular, if $T,S\in\mathbb{B}(X)$ satisfy $v(T)v(S)=0$, then $v(T+S)\geq |v(T)-v(S)|=v(T)+v(S)$ and hence $T \parallel_{v} S$.
\end{remark}

If $T, S \in \mathbb{B}(X)$ are linearly dependent, say $S=\alpha T$ for some $\alpha\neq0$. Then
$$v\left(T+\frac{\overline{\alpha}}{|\alpha|}S\right)=v(T+|\alpha|T)= (1+|\alpha|)v(T)=v(T)+v(S).$$
Consequently, $T \parallel_{v} S$. However, the converse of this statement is invariably false, as established in the following theorem.

\begin{theorem}\label{linear dependence}
	Let $X$ be a normed space of dimension at least 2. Then the numerical radius parallelism in $\mathbb{B}(X)$ dose not coincide with the linear dependence.
\end{theorem}

To establish this theorem, we need the following proposition.
\begin{Proposition}\label{PvI}
	Let $T \in \mathbb{B}(X)$. Then $T \parallel_{v} \alpha I$ for all $\alpha\in\mathbb{F}$.
\end{Proposition}
\begin{proof}
	Without loss of generality, we assume that $\alpha\neq0$.
	By the definition of $v(T)$, there exist sequences $\{x_n\}\subset S_X$ and $\{x^*_n\}\subset S_{X^*}$ with $x^*_n(x_n)=1$ for all $n\in\mathbb{Z}_{\geq1}$ such that
	$$\lim_{n\to\infty}\big|x^*_n\left(Tx_n\right)\big|=v(T).$$
	For every $n\in\mathbb{Z}_{\geq1}$, there exists $\lambda_n\in\mathbb{T}$ such that
	$$x^*_n\left(Tx_n\right)=\lambda_n\big|x^*_n\left(Tx_n\right)\big|.$$
	By passing to a subsequence, we may assume that $\{\lambda_n\}$ converges to some $\lambda\in\mathbb{T}$. Thus
	$$\lim_{n\to\infty}x^*_n\left(Tx_n\right)=\lambda v(T).$$
	Observe that $v(\alpha I)=|\alpha|$, $\lambda\overline{\alpha}/|\alpha|\in\mathbb{T}$ and
\begin{align*}
	&\big|x^*_n\left[\left(T+\lambda\overline{\alpha}/|\alpha|\cdot\alpha I\right)x_n\right]\big|\\
	=&\big|\lambda_n|x^*_n\left(Tx_n\right)|+\lambda|\alpha|\big|\to\big|\lambda v(T)+\lambda v(\alpha I)\big|=v(T)+v(\alpha I)
\end{align*}
	As a result,
	$$v\left(T+\lambda\overline{\alpha}/|\alpha|\cdot\alpha I\right)\geq v(T)+v(\alpha I).$$
	Consequently, $T \parallel_{v} \alpha I$.
\end{proof}

\begin{proof}[Proof of Theorem \ref{linear dependence}]
	Let $x\in S_X$, $x^*\in J(x)$ and $T=x^*(\cdot)x$. We can apply the previous proposition to obtain $T\parallel_{v} I$.
	However, it is evident that $T$ and $I$ are linearly independent.
\end{proof}

In the subsequent example, we shall demonstrate that the relations $\parallel_{v}$ and $\parallel$ are in general not comparable.
\begin{example}
	Let $X=\mathbb{R}^2$ equipped with the norm defined by $\|(x,y)\|=(x^4+y^4)^{1/4}$ for $x,y\in \mathbb{R}$.
	Consider
\begin{align*}
	&T:X\rightarrow X,\;(x,y)\mapsto(0,x),\\
	&S:X\rightarrow X,\;(x,y)\mapsto(y,x).
\end{align*}
    Clearly, 
    $$\|T\|=\|S\|=1,$$
    $$\|T+S\|=2,$$
    which implies that $T\parallel S$. We will show that $T\nparallel_{v} S$. Actually, for any $z^*\in S_{X^*}$, there exist $\alpha,\beta\in\mathbb{R}$ such that
    $$\alpha^{4/3}+\beta^{4/3}=1$$
    and
    $$z^*(x,y)=\alpha x+\beta y$$
    for all $(x,y)\in X$.
    Given that, in addition, $(x,y)\in S_X$ and $z^*(x,y)=1$, we arrive at the following equation:
    $$\alpha x+\beta y=1=(\alpha^{4/3}+\beta^{4/3})^{3/4}(x^4+y^4)^{1/4}.$$
    Pursuant to the equality conditions of the H\"{o}lder's inequality, we further deduce
    $$|\alpha|^{4/3} |y|^4=|\beta|^{4/3} |x|^4,$$
    and hence
    $$\alpha=x^3,\;\beta=y^3.$$
This leads to
\begin{align*}
	v(T)&=\sup\{|z^*(Tz)|:z\in S_X,z^*\in J(z)\}\\
	&=\sup\{|\beta x|:x^4+y^4=1,\alpha^{4/3}+\beta^{4/3}=1,\alpha x+\beta y=1\}\\
	&=\max\{xy^3:x^{4}+y^{4}=1\},\\
	v(S)&=\sup\{|z^*(Sz)|:z\in S_X,z^*\in J(z)\}\\
	&=\sup\{|\alpha y+\beta x|:x^4+y^4=1,\alpha^{4/3}+\beta^{4/3}=1,\alpha x+\beta y=1\}\\
	&=\max\{yx^3+xy^3:x^{4}+y^{4}=1\},\\
	v(T+S)&=\sup\{|z^*\left[(T+S)z\right]|:z\in S_X,z^*\in J(z)\}\\
	&=\sup\{|\alpha y+2\beta x|:x^4+y^4=1,\alpha^{4/3}+\beta^{4/3}=1,\alpha x+\beta y=1\}\\
	&=\max\{yx^3+2xy^3:x^{4}+y^{4}=1\},\\
	v(T-S)&=\sup\{|z^*\left[(T-S)z\right]|:z\in S_X,z^*\in J(z)\}\\
	&=\sup\{|\beta x-\alpha y|:x^4+y^4=1,\alpha^{4/3}+\beta^{4/3}=1,\alpha x+\beta y=1\}\\
	&=\max\{yx^3:x^{4}+y^{4}=1\}=\max\{xy^3:x^{4}+y^{4}=1\}.
\end{align*}
    Under the condition $x^4+y^4=1$, we have
\begin{align*}
	&(xy^3)^{1/4}=\left[3^{3/4}x\left(\frac{y}{3^{1/4}}\right) ^3\right]^{1/4}\leq\frac{y}{3^{1/16}}\left[\frac{x^4+3\left(\frac{y}{3^{1/4}}\right)^4}{4}\right]^{1/4}=\left(\frac{3^{3/4}}{4}\right)^{1/4},\\
	&(yx^3+xy^3)^{1/2}=\sqrt{2}\sqrt{xy}\left(\frac{x^2+y^2}{2}\right)^{1/2}\leq\sqrt{2}\left(\frac{x^4+y^4}{2}\right)^{1/4} \left(\frac{x^4+y^4}{2}\right)^{1/4}=1,
\end{align*}
    with equality in the first inequality if and only if $3^{1/4}x=y$ and equality in the second inequality if and only if $x=y$. Consequently,
    $$yx^3+2xy^3=(yx^3+xy^3)+xy^3<1+\frac{3^{3/4}}{4},$$
    since equality in both inequalities cannot be achieved simultaneously under the given condition.
    As a result
    $$v(T\pm S)<v(S)+v(T),$$
    leading to $T\nparallel_{v} S$.
    
    On the other hand, 
    $$\|T\pm I\|=\max\left\{[x^4+(x\pm y)^4]^{1/4}:x^4+y^4=1\right\}<2=\|T\|+\|I\|.$$
    In fact, under the condition $x^4+y^4=1$, we have
\begin{align*}
	&x^4+(x\pm y)^4=1+(x\pm y)^4-y^4\\
    =&\left\{ \begin{array}{cl}
		1+x(x+2y)[(x+y)^2+y^2]<16, & |x|<1,\\
		1, & |x|=1.\\
	\end{array}\right.
\end{align*}
    Thus $T\nparallel I$, despite $T\parallel_{v} I$ as established in Proposition \ref{PvI}.
\end{example}

One of our main results is the following characterization of numerical radius parallelism in $\mathbb{B}(X)$.

\begin{theorem}\label{v1}
	Let $T, S \in \mathbb{B}(X)$. Then the following conditions are equivalent:  
	
	$\mathrm{(i)}$ $T \parallel_{v} S$.  

	$\mathrm{(ii)}$ There exist sequences $\{x_n\}\subset S_X$ and $\{x^*_n\}\subset S_{X^*}$ with $x^*_n(x_n)=1$ for all $n\in\mathbb{Z}_{\geq1}$ such that
    $$\lim_{n\to\infty}\big|x^*_n\left(Tx_n\right)\big|\big|x^*_n\left(Sx_n\right)\big|=v(T)v(S).$$

	$\mathrm{(iii)}$ There exist sequences $\{x_n\}\subset S_X$ and $\{x^*_n\}\subset S_{X^*}$ with $x^*_n(x_n)=1$ for all $n\in\mathbb{Z}_{\geq1}$ such that
	$$\lim_{n\to\infty}\big|x^*_n\left(Tx_n\right)\big|=v(T)\quad \text{and} \quad\lim_{n\to\infty}\big|x^*_n\left(Sx_n\right)\big|=v(S).$$

\end{theorem}
\begin{proof}
	$(\mathrm{i})\Rightarrow(\mathrm{ii})$: Suppose that there exists $\lambda\in\mathbb{T}$ such that
	$$v(T+\lambda S)=v(T)+v(S).$$
	By the definition of the numerical radius of $T+\lambda S$, there exist $\{x_n\}\subset S_X$ and $\{x^*_n\}\subset S_{X^*}$ with $x^*_n(x_n)=1$ for all $n\in\mathbb{Z}_{\geq1}$ such that
	$$\lim_{n\to\infty}\big|x^*_n\left[(T+\lambda S)x_n\right]\big|=v(T+\lambda S).$$
	Then the following sequence of inequalities hold:
\begin{align*}
	v(T)\geq\big|x^*_n\left(Tx_n\right)\big|
	\geq&\big|x^*_n\left(Tx_n\right)\big|+\big|x^*_n\left(Sx_n\right)\big|-v(S)\\
	\geq&\big|x^*_n\left[(T+\lambda S)x_n\right]\big|-v(S)\\
	\to &v(T+\lambda S)-v(S)=v(T),\\
	v(S)\geq\big|x^*_n\left(Sx_n\right)\big|
	\geq&\big|x^*_n\left(Tx_n\right)\big|+\big|x^*_n\left(Sx_n\right)\big|-v(T)\\
	\geq&\big|x^*_n\left[(T+\lambda S)x_n\right]\big|-v(T)\\
	\to &v(T+\lambda S)-v(T)=v(S).
\end{align*}
	Thus
\begin{align}\label{F1}
	\lim_{n\to\infty}\big|x^*_n\left(Tx_n\right)\big|=v(T)\quad \text{and} \quad\lim_{n\to\infty}\big|x^*_n\left(Sx_n\right)\big|=v(S)
\end{align}
    hold. Consequently,
    $$\lim_{n\to\infty}\big|x^*_n\left(Tx_n\right)\big|\big|x^*_n\left(Sx_n\right)\big|=v(T)v(S).$$
    
    $(\mathrm{ii})\Rightarrow(\mathrm{iii})$: If $v(T)=0$ or $v(S)=0$, the proof reduces to a straightforward case. Therefore, we proceed under the assumption that $v(T)v(S)\neq0$.
    Suppose that there exist sequences $\{x_n\}\subset S_X$ and $\{x^*_n\}\subset S_{X^*}$ with $x^*_n(x_n)=1$ for all $n\in\mathbb{Z}_{\geq1}$ such that
    $$\lim_{n\to\infty}\big|x^*_n\left(Tx_n\right)\big|\big|x^*_n\left(Sx_n\right)\big|=v(T)v(S).$$
    Then
    \begin{align*}
    	v(T)\geq\big|x^*_n\left(Tx_n\right)\big|\geq\frac{\big|x^*_n\left(Tx_n\right)\big|\big|x^*_n\left(Sx_n\right)\big|}{v(S)}\to\frac{v(T)v(S)}{v(S)}=v(T),\\
    	v(S)\geq\big|x^*_n\left(Sx_n\right)\big|\geq\frac{\big|x^*_n\left(Tx_n\right)\big|\big|x^*_n\left(Sx_n\right)\big|}{v(T)}\to\frac{v(T)v(S)}{v(T)}=v(S),
    \end{align*}
    implying (\ref{F1}) hold.
    
    $(\mathrm{iii})\Rightarrow(\mathrm{i})$: Without loss of generality, we assume that $v(T)v(S)\neq0$. Suppose that there exist $\{x_n\}\subset S_X$ and $\{x^*_n\}\subset S_{X^*}$ with $x^*_n(x_n)=1$ for all $n\in\mathbb{Z}_{\geq1}$ such that (\ref{F1}) hold. 
    By passing to a subsequence, we may assume that there exist $\mu_1,\mu_2\in\mathbb{T}$ such that
    $$x^*_n\left(Tx_n\right)\to v(T)\mu_1,\quad x^*_n\left(Sx_n\right)\to v(S)\mu_2.$$
    Let $\lambda=\mu_1\overline{\mu_2}$, then $\lambda\in\mathbb{T}$ and
\begin{align*}
	x^*_n\left[(T+\lambda S)x_n\right]=x^*_n\left(Tx_n\right)+\mu_1\overline{\mu_2}x^*_n\left(Sx_n\right)
	\to v(T)\mu_1+v(S)\mu_1.
\end{align*}
    It follows that
    $$v(T+\lambda S)\geq \lim_{n\to\infty}\big|x^*_n\left[(T+\lambda S)x_n\right]\big|=v(T)+v(S).$$
    This leads to $T \parallel_{v} S$.    
\end{proof}
\begin{remark}
	If $X$ is a Hilbert space equipped with the inner product $\left\langle \cdot,\cdot \right\rangle$, then $x\in S_X$ and $x^*\in J(x)$ imply that $x^*=\left\langle \cdot,x \right\rangle$. Consequently, by applying the current theorem, we are able to reestablish Theorem \ref{Mehrazin}.
\end{remark}

\begin{corollary}
	Let $X$ be a finite dimensional normed space and $T, S \in \mathbb{B}(X)$. Then the following conditions are equivalent:  
	
	$\mathrm{(i)}$ $T \parallel_{v} S$.  

    $\mathrm{(ii)}$ There exist $x\in S_X$ and $x^*\in J(x)$ such that
    $$\big|x^*\left(Tx\right)\big|\big|x^*\left(Sx\right)\big|=v(T)v(S).$$
	
	$\mathrm{(iii)}$ There exist $x\in S_X$ and $x^*\in J(x)$ such that
	$$\big|x^*\left(Tx\right)\big|=v(T)\quad \text{and} \quad\big|x^*\left(Sx\right)\big|=v(S).$$	
\end{corollary}
\begin{proof}
	$(\mathrm{iii})\Rightarrow(\mathrm{ii})$: Clearly.
    
    $(\mathrm{ii})\Rightarrow(\mathrm{i})$: It follows immediately from Theorem \ref{v1}.
	
	$(\mathrm{i})\Rightarrow(\mathrm{iii})$: Suppose that $T \parallel_{v} S$. By Theorem \ref{v1}, there exist sequences $\{x_n\}\subset S_X$ and $\{x^*_n\}\subset S_{X^*}$ with $x^*_n(x_n)=1$ for all $n\in\mathbb{Z}_{\geq1}$ such that
	$$\lim_{n\to\infty}\big|x^*_n\left(Tx_n\right)\big|=v(T)\quad \text{and} \quad\lim_{n\to\infty}\big|x^*_n\left(Sx_n\right)\big|=v(S).$$
	
	Since $X$ is finite dimensional, by passing to a subsequence, we can assume that $\{x_n\}$ and $\{x^*_n\}$ converge to some $x\in S_X$ and some $x^*\in S_{X^*}$, respectively. Then
\begin{align*}
	|x^*(x)-1|&=|x^*(x)-x^*_n(x_n)|\\
	&\leq|x^*(x-x_n)|+|(x^*-x^*_n)(x_n)|\\
	&\leq\|x-x_n\|+\|x^*-x^*_n\|\to0,\\
	\big||x^*(Tx)|-v(T)\big|&\leq\big||x^*(Tx)|-|x^*_n(Tx_n)|\big|+\big||x^*_n(Tx_n)|-v(T)\big|\\
	&\leq|x^*(Tx-Tx_n)|+|(x^*-x^*_n)(Tx_n)|+\big||x^*_n(Tx_n)|-v(T)\big|\\
	&\leq\|T\|\|x-x_n\|+\|T\|\|x^*-x^*_n\|+\big||x^*_n(Tx_n)|-v(T)\big|\to0,\\
	\big||x^*(Sx)|-v(S)\big|&\leq\big||x^*(Sx)|-|x^*_n(Sx_n)|\big|+\big||x^*_n(Sx_n)|-v(S)\big|\\
	&\leq|x^*(Sx-Sx_n)|+|(x^*-x^*_n)(Sx_n)|+\big||x^*_n(Sx_n)|-v(S)\big|\\
	&\leq\|S\|\|x-x_n\|+\|S\|\|x^*-x^*_n\|+\big||x^*_n(Sx_n)|-v(S)\big|\to0.
\end{align*}
    Consequently, $x^*\in J(x)$, $\big|x^*\left(Tx\right)\big|=v(T)$ and $\big|x^*\left(Sx\right)\big|=v(S)$ hold.
\end{proof}

Concisely stated, based on the definition of numerical radius parallelism or Theorem \ref{v1}, fundamental properties are derived immediately and listed below without proof.

\begin{Proposition}
	$ $
	$\mathrm{(i)}$ Numerical radius parallelism possesses symmetry, that is, if $T,S\in \mathbb{B}(X)$ satisfy $T \parallel_{v} S$, then $S \parallel_{v} T$.
	
	$\mathrm{(ii)}$ Numerical radius parallelism possesses homogeneity, which means that if $T,S\in \mathbb{B}(X)$ satisfy $T \parallel_{v} S$, then $\alpha T \parallel_{v}\beta S$ for all $\alpha,\beta\in\mathbb{F}$.
\end{Proposition}
 
A natural question then arises regarding its transitivity, i.e., whether the following implication is true: 
    $$T,S,R\in\mathbb{B}(X)\setminus\{0\},\;T\parallel_{v}S,\;S\parallel_{v}R\;\Rightarrow\;T\parallel_{v}R.$$
The answer is negative. Mehrazin $\mathit{et\;al.}$ \cite[Example]{Mehrazin} presented an example to show that the numerical radius parallelism of Hilbert space operators is in general not transitive. We will establish a more thorough result.
\begin{theorem}\label{Tpt}
	Let $X$ be a normed space with dimension not less than 2. Then the numerical radius parallelism in $\mathbb{B}(X)$ fails to be transitive.
\end{theorem}
To this end, we need the following proposition, which provides a sufficient condition for norm parallelism in terms of numerical radius parallelism.
\begin{Proposition}\label{Pnv}
	Suppose that $x,y \in X$, $x^*\in J(x)$ and $y^* \in J(y)$ such that $x^*(\cdot)x \parallel_{v} y^*(\cdot)y$ holds. Then $x \parallel y$.
\end{Proposition}
\begin{proof}
	Without losing generality, we assume that $\|x\|\|y\|\neq0$.
	Let $T=x^*(\cdot)x$ and $S=y^*(\cdot)y$. Then
\begin{align*}
	v(T)=&\sup\{|z^*\left[x^*(z)x\right]|:z\in S_X,z^*\in S_{X^*}, z^*(z)=1\}\leq\|x\|,\\
	v(T)\geq&x^*\left[x^*\left(\frac{x}{\|x\|}\right) x\right]=\|x\|.
\end{align*}
    It follows that $v(T)=\|x\|$. Similarly, $v(S)=\|y\|$.
	By Theorem \ref{v1}, there exist sequences $\{z_n\}\subset S_X$ and $\{z^*_n\}\subset S_{X^*}$ with $z^*_n(z_n)=1$ for all $n\in\mathbb{Z}_{\geq1}$ such that
	$$\lim_{n\to\infty}\big|z^*_n\left[x^*(z_n)x\right]\big|=\|x\|\quad \text{and} \quad\lim_{n\to\infty}\big|z^*_n\left[y^*(z_n)y\right]\big|=\|y\|.$$
	Observe that
    $$\|x\|=\lim_{n\to\infty}\big|z^*_n\left[x^*(z_n)x\right]\big|\leq\liminf_{n\to\infty}\big|z^*_n(x)\big|\leq\limsup_{n\to\infty}\big|z^*_n(x)\big|\leq\|x\|.$$
    Thus $\big|z^*_n(x)\big|\to\|x\|$. Similarly, $\big|z^*_n(y)\big|\to\|y\|$. By passing to a subsequence, we may assume that there exist $\mu_1,\mu_2\in\mathbb{T}$ such that
    $$z^*_n(x)\to \|x\|\mu_1,\quad z^*_n(y)\to \|y\|\mu_2.$$
    Let $\lambda=\mu_1\overline{\mu_2}$, then $\lambda\in\mathbb{T}$ and
    $$\|x\|+\|y\|\geq\|x+\lambda y\|\geq|z^*_n(x)+\lambda z^*_n(y)|\to\big|\|x\|\mu_1+\|y\|\mu_1\overline{\mu_2}\mu_2\big|=\|x\|+\|y\|.$$
    As a result, $\|x+\lambda y\|=\|x\|+\|y\|$, leading to $x \parallel y$.
\end{proof}

Recall that a point $x\in X$ is said to be a smooth point if its duality mapping $J(x)$ is a singleton. It is worth noting that, as a particular instance of Rademacher's theorem \cite{Rademache}, almost every point (in the sense of Lebesgue measure) in a finite dimensional normed space is a smooth point. For further reference, see \cite[Proposition 2.3]{Blanco}.
\begin{proof}[Proof of Theorem \ref{Tpt}]
	Let $Y$ be a 2 dimensional subspace of $X$, $x\in S_Y$ a smooth point of $Y$, $f\in S_{Y^*}$ with $f(x)=1$ and $y\in S_Y$ satisfies $f(y)=0$. 
	We claim that $x\nparallel y$. If not, then there exists $\lambda\in\mathbb{T}$ such that
	$$\|x+\lambda y\|=\|x\|+\|y\|=2.$$
	It follows that there exists $g\in S_{Y^*}$ such that $$2=g\left(x+\lambda y\right)=g(x)+\lambda g(y).$$
	This leads to $g(x)=1$ and $\left|g(y)\right|=1$. Since $x$ is a smooth point of $Y$, $f=g$. But then $g(y)=0$, a contradiction. The Hahn-Banach theorem yields $x^*,y^* \in S_{X^*}$ such that $x^*(x)=y^*(y)=1$. Applying Proposition \ref{Pnv}, we arrive at $x^*(\cdot)x$ is not numerical radius parallel to $y^*(\cdot)y$. Since $x^*(\cdot)x\parallel_{v} I$ and $I\parallel_{v}y^*(\cdot)y$
	hold as established in Proposition \ref{PvI}, we conclude that the numerical radius parallelism in $\mathbb{B}(X)$ is not transitive.
\end{proof}

It is noteworthy that the converse of Proposition \ref{Pnv} holds true provided that $X$ is a complex Hilbert space, as demonstrated in \cite[Corollary 2.6]{Mehrazin}. In this context, we proceed to present some weaker conditions that suffice for the converse of Proposition \ref{Pnv} to be valid. 

The subsequent theorem focus on \emph{rotund points}, defined as elements of $S_X$ that do not lie within any line segment contained in $S_X$. It is worth recalling that $X$ is strictly convex (or smooth, resp.) precisely when $S_X$ consists entirely of rotund points (or smooth points, resp.).
\begin{theorem}
	Let $X$ be a normed space with dimension not less than 2. Then the implication $\mathrm{(i)}\,\Rightarrow\,\mathrm{(ii)}$ holds.
	
	$\mathrm{(i)}$ Every point $x\in S_X$ is a rotund point or a smooth point in any 2-dimensional subspace including $x$.
	
	$\mathrm{(ii)}$ If $x,y\in X$ satisfy $x\parallel y$. Then for any $x^*\in J(x)$ and $y^*\in J(y)$, the condition $x^*(\cdot)x \parallel_{v} y^*(\cdot)y$ holds.
	
	In particular, if $X$ is strictly convex or smooth, then $\mathrm{(ii)}$ hold.
\end{theorem}
\begin{proof}
	Suppose that $\mathrm{(i)}$ holds and that $x,y\in X$ satisfy $x\parallel y$. By the homogeneity of $\parallel$ and $\parallel_v$, we may assume that $x,y\in S_X$. Then there exists $\lambda\in\mathbb{T}$ such that $\|x+\lambda y\|=2$.
	
	Case 1: If $x$ and $y$ are linearly independent. Then $x$ can not be a rotund point of $\mathrm{\mathop{span}}\{x,y\}$. Otherwise we have $x=\lambda y$, a contradiction. Thus $x$ is a smooth point of $\mathrm{\mathop{span}}\{x,y\}$. Let $f\in J(x+\lambda y)$. Then
	$$2=f(x+\lambda y)=f(x)+f(\lambda y)\leq |f(x)|+|f(\lambda y)|\leq 2.$$
	This leads to
	$f(x)=f(\lambda y)=1$. Since $x$ is a smooth point of $\mathrm{\mathop{span}}\{x,y\}$, we arrive at $x^*=f$ on $\mathrm{\mathop{span}}\{x,y\}$. Hence
\begin{align*}
	v\big(x^*(\cdot)x+y^*(\cdot)y\big)\geq&\big|x^*(\lambda y)f(x)+y^*(\lambda y)f(y)\big|\\
	=&\big|f(\lambda y)f(x)+y^*(y)f(\lambda y)\big|\\
	=&2=v\big(x^*(\cdot)x\big)+v\big(y^*(\cdot)y\big),
\end{align*}
    which entails that $x^*(\cdot)x \parallel_{v} y^*(\cdot)y$.
	
	Case 2: If $x$ and $y$ are linearly dependent, then there exists $\mu\in\mathbb{T}$ such that $\mu x=y$. It follows that
\begin{align*}
	&v\big(x^*(\cdot)x+y^*(\cdot)y\big)\geq\big|x^*(x)x^*(x)+y^*(x)x^*(y)\big|\\
	=&\big|1+y^*(x)x^*(\mu x)\big|=\big|1+y^*(\mu x)x^*(x)\big|\\
	=&\big|1+y^*(y)x^*(x)\big|=2=v\big(x^*(\cdot)x\big)+v\big(y^*(\cdot)y\big),
\end{align*}
    which implies that $x^*(\cdot)x \parallel_{v} y^*(\cdot)y$.  
\end{proof}
\begin{remark}
	The implication $\mathrm{(ii)}\,\Rightarrow\,\mathrm{(i)}$ is invalid. To illustrate this, consider the space $X=\mathbb{R}^2$ equipped with the norm of element $(x_1,x_2)$ defined by
	$$\|(x_1,x_2)\|:=\left\{ \begin{array}{cl}
		x_1^2+x_2^2 & x_1x_2\geq0,\\
		\max\{|x_1|,|x_2|\} & x_1x_2\leq0,\\
	\end{array}\right.$$
    Then $y=(-1,1)\in S_X$ is neither a rotund point nor a smooth point. Observe that every point in $S_X$ that is distinct from $\pm y$ is a smooth point. By analogy with the proof of the previous theorem, it can be shown that $\mathrm{(ii)}$ holds $($note that in Case 1 of the proof, the sole requirement is that $x$ be a smooth point$)$.
    
    Nevertheless, we present an example to show that $\mathrm{(ii)}$ in the preceding theorem may fail in general.
\end{remark}
\begin{example}
	Let $X$ be the $\mathbb{F}$-vector space $\mathbb{F}^2$, where $\mathbb{F}=\mathbb{R}$ or $\mathbb{C}$, equipped with the norm of element $(x_1,x_2)$ defined by $\|(x_1,x_2)\|=|x_1|+|x_2|$. Let $x=(1,0)$ and $y=(0,1)$. Then $x\parallel y$ since
	$$\|(1,0)+(0,1)\|=2=\|(1,0)\|+\|(0,1)\|.$$
	Identify $X^*$ with $\mathbb{F}^2$ equipped with the norm of element $(x_1,x_2)$ defined by $\|(x_1,x_2)\|=\max\{|x_1|,|x_2|\}$. Let $x^*=(1,0)$ and $y^*=(0,1)$. Then $x^*\in J(x)$ and $y^*\in J(y)$. We claim that
	$x^*(\cdot)x \nparallel_{v} y^*(\cdot)y$. If not, then there exist $z=(z_1,z_2)\in S_X$, $z^*=(a,b)\in J(z)$ and $\lambda\in\mathbb{T}$ such that
\begin{align*}
	2&\geq|az_1|+|bz_2|=|az_1+\lambda bz_2|=|x^*(z)z^*(x)+\lambda y^*(z)y^*(x)|\\
	&=v\big(x^*(\cdot)x+\lambda y^*(\cdot)y\big)=v\big(x^*(\cdot)x\big)+v\big(y^*(\cdot)y\big)=2,
\end{align*}
    leading to $|az_1|=|bz_2|=1$. Thus $|z_1|\geq|az_1|=1$ and $|z_2|\geq|bz_2|=1$, contradicts to $|z_1|+|z_2|=1$.
\end{example}

\section{Numerical radius Birkhoff orthogonality}\label{s3}
Mal $\mathit{et\;al.}$ \cite{Mal2} introduced the concept of numerical radius Birkhoff orthogonality for Hilbert space operators, as follows.
\begin{definition}\cite[Definition 1.1 and 1.2]{Mal2}
	Let $H$ be a Hilbert space.
	An element $T \in \mathbb{B}(H)$ is called \emph{numerical radius Birkhoff orthogonal} to another  element $S \in \mathbb{B}(H)$, denoted by $T \perp_{\omega B} S$, if $\omega(T+\alpha S)\geq \omega(T)$ for all $\alpha\in\mathbb{F}$.
\end{definition}

As an application, they derive numerical radius inequalities for bounded linear operators on a complex Hilbert space, yielding improved lower bounds for the numerical radius of a matrix compared to existing estimates.

We now extend the concept of numerical radius Birkhoff orthogonality, which was originally defined on $\mathbb{B}(H)$, to $\mathbb{B}(X)$.
\begin{definition}
	An element $T \in \mathbb{B}(X)$ is called \emph{numerical radius Birkhoff orthogonal} to another  element $S \in \mathbb{B}(X)$, denoted by $T \perp_{vB} S$, if $v(T+\alpha S)\geq v(T)$ for all $\alpha\in\mathbb{F}$.
\end{definition}
\begin{remark}
	Since $A\mapsto v(A)$ is a semi-norm on $\mathbb{B}(X)$, the condition $T \perp_{vB} S$ holds whenever $v(T)v(S)=0$.
\end{remark}

The following proposition follows directly form the definition of the numerical radius Birkhoff orthogonality.
\begin{Proposition}
	Numerical radius Birkhoff orthogonality possesses homogeneity, that is, if $T,S\in \mathbb{B}(X)$ satisfy $T \perp_{vB} S$, then $\alpha T \perp_{vB}\beta S$ for all $\alpha,\beta\in\mathbb{F}$.
\end{Proposition}

A natural question then arises regarding its left additivity and right additivity, i.e., whether the following implications are true:
$$T\perp_{vB}R,\;S\perp_{vB}R\;\Rightarrow\;T+S\perp_{vB}R,$$
$$T\perp_{vB}S,\;T\perp_{vB}R\;\Rightarrow\;T\perp_{vB}S+R.$$
The answer is negative. More precisely, we will establish the following theorem.
\begin{theorem}\label{pva}
	Let $X$ be a normed space of dimension at least 2. Then the numerical radius Birkhoff orthogonality in $\mathbb{B}(X)$ is neither additive on the left nor on the right.
\end{theorem}

To establish this theorem, we need the following proposition.
\begin{Proposition}
	Let $X$ be a normed space. Then $I\perp_{vB}T$ whenever $T\in\mathbb{B}(X)$ is not injective.
\end{Proposition}
\begin{proof}
	There exists $x\in S_X$ such that $Tx=0$. Let $x^*\in J(x)$. Then for any $\alpha\in\mathbb{F}$, we have
	$$v\left(I+\alpha T\right)\geq x^*\left[\left(I+\alpha T\right)(x)\right]=1=v(I),$$
	leading to the desired result.
\end{proof}

\begin{proof}[Proof of Theorem \ref{pva}]
	Let $x\in S_X$ and $x^*\in J(x)$. Clearly, $x^*(\cdot)x$ and $I-x^*(\cdot)x$ are not injective. We can apply the previous proposition to arrive at
	$$I\perp_{vB}x^*(\cdot)x,\quad I\perp_{vB}I-x^*(\cdot)x.$$
	Since $v(I-I)=0<1=v(I)$, i.e., $I\not\perp_{vB}I$, we conclude that the numerical radius Birkhoff orthogonality in $\mathbb{B}(X)$ is not additive on the right. 
	
	On the other hand, suppose for contrary that the numerical radius Birkhoff orthogonality in $\mathbb{B}(X)$ is additive on the left. Observe that
\begin{align*}
	v\Big(-x^*(\cdot)x+\alpha\left[I-x^*(\cdot)x\right]\Big)&\geq\big|-x^*(x)x^*(x)+\alpha\left[x^*(x)-x^*(x)x^*(x)\right]\big|\\
	&=1=v\Big(-x^*(\cdot)x\Big)
\end{align*}
    for any $\alpha\in\mathbb{F}$. Thus $-x^*(\cdot)x\perp_{vB}I-x^*(\cdot)x$. Since $I\perp_{vB}I-x^*(\cdot)x$, we obtain
    $$I-x^*(\cdot)x\perp_{vB}I-x^*(\cdot)x.$$
    Notice that there exists $y\in S_X$ such that $x^*(y)=0$. Let $y^*\in J(y)$. However,
\begin{align*}
	v\Big(\left[I-x^*(\cdot)x\right]-\left[I-x^*(\cdot)x\right]\Big)=0<1=\big|y^*(y)-x^*(y)y^*(x)\big|\leq v\Big(I-x^*(\cdot)x\Big),
\end{align*}
    a contradiction.
\end{proof}

We now establish a characterization of numerical radius Birkhoff orthogonality in $\mathbb{B}(X)$, which generalizes the characterizations of numerical radius Birkhoff orthogonality for Hilbert space operators \cite[Theorem 2.3 and 2.5]{Mal2}.
\begin{theorem}\label{vo1}
	The condition $T \perp_{vB} S$ holds if and only if for any $\lambda \in \mathbb{T}$ there exist $\{x_{\lambda,n}\}_{n=1}^{\infty}\subset S_X$ and $\{x^{*}_{\lambda,n}\}_{n=1}^{\infty}\subset S_{X^*}$ with $x^*_{\lambda,n}(x_{\lambda,n})=1$ for all $n\in\mathbb{Z}_{\geq1}$, such that the following conditions hold:
	$$\lim_{n\to\infty}\big|x^{*}_{\lambda,n}\left(Tx_{\lambda,n}\right)\big|=v(T),\quad \lim_{n\to\infty}\mathrm{\mathop{Re}}\left[\lambda \overline{x^{*}_{\lambda,n}\left(Tx_{\lambda,n}\right)}x^{*}_{\lambda,n}\left(Sx_{\lambda,n}\right)\right]\geq0.$$
\end{theorem}
\begin{proof}
	Suppose that $T \perp_{vB} S$. Then for any $\lambda \in \mathbb{T}$ and all $n\in\mathbb{Z}_{\geq1}$, we have
	$$v\left(T+\frac{\lambda}{\sqrt{n}} S\right)\geq v(T).$$
	Observe that, for each $n\in\mathbb{Z}_{\geq1}$, there exist $x_{\lambda,n}\in S_X$ and $x^{*}_{\lambda,n}\in S_{X^*}$ such that $x^*_{\lambda,n}(x_{\lambda,n})=1$ and
	$$\left| x^{*}_{\lambda,n}\left[ \left(T+\frac{\lambda}{\sqrt{n}} S\right)(x_{\lambda,n})\right]\right|\geq \frac{n}{n+1}v\left(T+\frac{\lambda}{\sqrt{n}} S\right)\geq \frac{n}{n+1}v(T).$$
	It follows that 
\begin{align*}
	&\big|x^{*}_{\lambda,n}\left(Tx_{\lambda,n}\right)\big|\geq\left| x^{*}_{\lambda,n}\left[ \left(T+\frac{\lambda}{\sqrt{n}} S\right)(x_{\lambda,n})\right]\right|-\frac{\big|x^{*}_{\lambda,n}\left(Sx_{\lambda,n}\right)\big|}{\sqrt{n}}\\
	\geq&\frac{n}{n+1}v(T)-\frac{v(S)}{\sqrt{n}},\\
	&\mathrm{\mathop{Re}}\left[\lambda \overline{x^{*}_{\lambda,n}\left(Tx_{\lambda,n}\right)}x^{*}_{\lambda,n}\left(Sx_{\lambda,n}\right)\right]\\
	=&\frac{\sqrt{n}}{2}\left\lbrace \left| x^{*}_{\lambda,n}\left[ \left(T+\frac{\lambda}{\sqrt{n}} S\right)(x_{\lambda,n})\right]\right|^2-\big|x^{*}_{\lambda,n}\left(Tx_{\lambda,n}\right)\big|^2-\frac{1}{n}\big|x^{*}_{\lambda,n}\left(Sx_{\lambda,n}\right)\big|^2\right\rbrace\\
	\geq&\frac{\sqrt{n}}{2}\left\lbrace \left[\frac{n}{n+1}v(T)\right]^2 -v(T)^2-\frac{v(S)^2}{n}\right\rbrace=-\frac{(2n+1)\sqrt{n}}{2(n+1)^2}v(T)^2-\frac{v(S)^2}{2\sqrt{n}}.
\end{align*}
    Since the sequences $\left\{x^{*}_{\lambda,n}\left(Tx_{\lambda,n}\right)\right\}$ and $\left\{x^{*}_{\lambda,n}\left(Sx_{\lambda,n}\right)\right\}$are bounded, by passing to a subsequence, we can assume that they converge. Consequently,
    $$\lim_{n\to\infty}\big|x^{*}_{\lambda,n}\left(Tx_{\lambda,n}\right)\big|\geq v(T),\quad \lim_{n\to\infty}\mathrm{\mathop{Re}}\left[\lambda \overline{x^{*}_{\lambda,n}\left(Tx_{\lambda,n}\right)}x^{*}_{\lambda,n}\left(Sx_{\lambda,n}\right)\right]\geq0.$$
    Since $\lim_{n\to\infty}\big|x^{*}_{\lambda,n}\left(Tx_{\lambda,n}\right)\big|\leq v(T)$ holds, we establish the necessity.
    
    It remains to prove the sufficiency. For any $\alpha\in \mathbb{F}$, there exists $\lambda \in \mathbb{T}$ such that $\alpha=|\alpha|\lambda$. As a result,
\begin{align*}
	v(T+\alpha S)^2\geq&\limsup_{n\to\infty}\big|x^{*}_{\lambda,n}\left[ \left(T+|\alpha|\lambda S\right)(x_{\lambda,n})\right]\big|^2\\
	=&\limsup_{n\to\infty}\left\lbrace  \big|x^{*}_{\lambda,n}\left(Tx_{\lambda,n}\right)\big|^2+|\alpha|^2|x^{*}_{\lambda,n}\left(Sx_{\lambda,n}\right)\big|^2\right.\\
	&\left.+2|\alpha|\mathrm{\mathop{Re}}\left[\lambda \overline{x^{*}_{\lambda,n}\left(Tx_{\lambda,n}\right)}x^{*}_{\lambda,n}\left(Sx_{\lambda,n}\right)\right]\right\rbrace\\
	\geq&\lim_{n\to\infty}\big|x^{*}_{\lambda,n}\left(Tx_{\lambda,n}\right)\big|^2= v(T)^2,
\end{align*}
    i.e., $v(T+\alpha S)\geq v(T)$. Therefore, $T \perp_{vB} S$ holds.
\end{proof}

The following theorem, analogies to Theorem \ref{T-1}, provides a characterization of numerical radius parallelism in terms of numerical radius Birkhoff orthogonality.
\begin{theorem}
	Let $X$ be a normed space and $T,S\in\mathbb{B}(X)$. Then $T\parallel_{v}S$ if and only if there exists $\lambda\in\mathbb{T}$ such that $T\perp_{vB}v(S)T+\lambda v(T)S$.
\end{theorem}
\begin{proof}
	Without lose of generality, we assume that $v(T)v(S)\neq0$.
	
	Suppose that $T\parallel_{v}S$. Then by Theorem \ref{v1}, there exist sequences $\{x_n\}\subset S_X$ and $\{x^*_n\}\subset S_{X^*}$ with $x^*_n(x_n)=1$ for all $n\in\mathbb{Z}_{\geq1}$ such that
	$$\lim_{n\to\infty}\big|x^*_n\left(Tx_n\right)\big|=v(T)\quad \text{and} \quad\lim_{n\to\infty}\big|x^*_n\left(Sx_n\right)\big|=v(S).$$
	By passing to a subsequence, we may assume that $\{x^*_n\left(Tx_n\right)\}$ and $\{x^*_n\left(Sx_n\right)\}$ exist.
	Let
	$$\lambda=-\lim_{n\to\infty}\frac{v(S)x^*_n\left(Tx_n\right)}{v(T)x^*_n\left(Sx_n\right)}.$$
	Then $\lambda\in\mathbb{T}$. For every $\mu\in\mathbb{T}$ and $n\in\mathbb{Z}_{\geq1}$, let $x_{\mu,n}=x_n$ and $x^*_{\mu,n}=x^*_n$. It follows that
	$$\lim_{n\to\infty}x^{*}_{\mu,n}\left(v(S)T+\lambda v(T)S\right)\left(x_{\mu,n}\right)=0$$
	for all $\mu\in\mathbb{T}$. Thus
	$$\lim_{n\to\infty}\big|x^{*}_{\mu,n}\left(Tx_{\mu,n}\right)\big|=v(T),$$
	$$\lim_{n\to\infty}\mathrm{\mathop{Re}}\left[\mu \overline{x^{*}_{\mu,n}\left(Tx_{\mu,n}\right)}x^{*}_{\mu,n}\left(v(S)T+\lambda v(T)S\right)\left(x_{\mu,n}\right)\right]=0,$$
	for all $\mu\in\mathbb{T}$. Applying Theorem \ref{vo1}, we obtain $T\perp_{vB}v(S)T+\lambda v(T)S$.
	
	Conversely, suppose that there exists $\lambda\in\mathbb{T}$ such that $T\perp_{vB}v(S)T+\lambda v(T)S$. By Theorem \ref{vo1}, there exist sequences $\{x_n\}\subset S_X$ and $\{x^*_n\}\subset S_{X^*}$ with $x^*_n(x_n)=1$ for all $n\in\mathbb{Z}_{\geq1}$ such that
	$$\lim_{n\to\infty}\big|x^*_n\left(Tx_n\right)\big|=v(T),$$
	$$\lim_{n\to\infty}\mathrm{\mathop{Re}}\left[- \overline{x^{*}_{n}\left(Tx_{n}\right)}x^{*}_{n}\left(v(S)T+\lambda v(T)S\right)\left(x_{,n}\right)\right]\geq0.$$
	By passing to a subsequence, we may assume that $\{x^*_n\left(Tx_n\right)\}$ and $\{x^*_n\left(Sx_n\right)\}$ exist. Let 
	$$\mu=\lim_{n\to\infty}\frac{x^*_n\left(Tx_n\right)}{v(T)}.$$
	Then $\mu\in\mathbb{T}$ and
\begin{align*}
	&\mathrm{\mathop{Re}}\left[-\lambda\overline{\mu}v(T)^2\lim_{n\to\infty}x^*_n\left(Sx_n\right)-v(T)^2v(S)\right]\\
	=&\lim_{n\to\infty}\mathrm{\mathop{Re}}\left[- \overline{x^{*}_{n}\left(Tx_{n}\right)}x^{*}_{n}\left(v(S)T+\lambda v(T)S\right)\left(x_{,n}\right)\right]\geq0.
\end{align*}
    Thus
    $$v(S)\leq\mathrm{\mathop{Re}}\left[-\lambda\overline{\mu}\lim_{n\to\infty}x^*_n\left(Sx_n\right)\right]\leq \lim_{n\to\infty}\big|x^*_n\left(Sx_n\right)\big|\leq v(S),$$
    leading to
    $$\lim_{n\to\infty}\big|x^*_n\left(Sx_n\right)\big|=v(S).$$
    Consequently, we can invoke Theorem \ref{v1} to conclude that $T\parallel_{v}S$.
\end{proof}



$$\\$$
$\mathbf {Data~ Availability}$
$$\\$$
No data were used to support this study.
$$\\$$
$\mathbf{Conflicts ~of~ Interest}$
$$\\$$
The author(s) declare(s) that there is no conflict of interest regarding the publication of this paper.
$$\\$$
$\mathbf{Funding~ Statement}$
$$\\$$
This work was supported by the National Natural Science Foundation
of P. R. China (No.11971493 and 12071491).

\bibliographystyle{plain}

\bibliography{2024-11-14}

\end{document}